\documentclass[12pt]{amsart}
\usepackage{extsizes}
\usepackage{amssymb,amsfonts,amsthm,amsmath}
\usepackage{graphicx}
\usepackage{color}
\usepackage[mathscr]{eucal} 
\usepackage{subcaption}
\usepackage{mathtools}
\usepackage[pagebackref=true]{hyperref}
\usepackage{hyperref}
    \hypersetup{
          unicode   = true,
          colorlinks= true,
    }
\usepackage{cite}

\AtBeginDocument{
\addtolength\abovedisplayskip{-0.25\baselineskip}
 \addtolength\belowdisplayskip{-0.25\baselineskip}
 \addtolength\abovedisplayshortskip{-0.25\baselineskip}
  \addtolength\belowdisplayshortskip{-0.25\baselineskip}
}
\usepackage[left=1 in,top=1 in,right=1 in, bottom=1 in]{geometry}
\usepackage[T1]{fontenc} 
\numberwithin{equation}{section}
\definecolor{darkred}{rgb}{.70,.12,.20}

\definecolor{darkgreen}{rgb}{.20,.52,.14}

\definecolor{byz}{rgb}{.44,.16,.39}

\numberwithin{equation}{section}
\DeclareMathOperator{\supp}{supp}
\newtheorem{lemma}{Lemma}
\newtheorem{remark}{Remark}
\newtheorem{definition}{Definition}
\newtheorem{corollary}{Corollary}
\newtheorem{assumption}{Assumption}
\newtheorem{theorem}{Theorem}

\newtheorem{axiom}{Axiom}
\newtheorem{proposition}{Proposition}
\newtheorem*{Axiom}{Axiom of Mass conservation}
\newtheorem{corollary }{Corollary}
\newtheorem{Lad-Ur}{Ladyzhenskaya-Uraltceva iterative Lemma}
\newtheorem{hypotheses}{Hypotheses}
\newtheorem{exmpl}{Example}
\usepackage{todonotes}

\title[Iterative Energy Estimate - Degenerate Einstein Brownian model]{An Iterative Energy Estimate for Degenerate Einstein model  of Brownian motion }
\author{
Isanka Garli Hevage$^1$, Akif Ibraguimov$^2$, Zeev Sobol$^{3}$
}
\date{}
\begin{document}
\maketitle
\begin{center}
{$^1$ Department of Mathematics and Statistics, Texas Tech University\\
Lubbock, Texas, USA, e-mail: \texttt{isankaupul.garlihevage@ttu.edu}
\smallskip
\\
{$^2$ Department of Mathematics and Statistics, Texas Tech University,}
\\
\small{Lubbock, Texas, USA, e-mail: \texttt{akif.ibraguimov@ttu.edu}}
\smallskip
\\
{$^3$ Department of Mathematics, University of Swansea,}
\\
{Fabian Way, Swansea SA1 8EN, UK, e-mail: \texttt{z.sobol@swansea.ac.uk}}
}
\end{center}
\begin{abstract}
\noindent
We consider the degenerate Einstein's Brownian motion model for the case when the time interval ($\tau$)  of particle Jumps before collision (free jumps) reciprocal to the number of particles per unit volume $u(x,t) > 0$ at the point of observation $x$ at time $t$. The parameter $0 < \tau \leq C < \infty$, controls characteristic of the fluid  "almost decreases"  to $ 0 $ when $u \rightarrow \infty$. This degeneration leads to the localisation of the spread of particle propagation in the media. In our report we will present a structural condition of the time interval of free jumps - $\tau$ and the frequency of these free jumps $\phi$ as  functions of $u$ which guarantees the finite speed of propagation of $u$. 
\end{abstract}
\section{Introduction}
Consider the sought experiment of the fluid which occupy a bounded domain $\Omega\subset \mathbb{R}^{N} \ ; \ N=1,2,\dots$  by the particles. Let $\vec{\Delta}$ be the $N$ dimensional vector of the free jumps which we  define  as the particles' jumps without collisions. Let  $(\Delta_{1},\Delta_{2}, \cdots,\Delta_{N})$ are the coordinates of free jump.In definition of free jump we follow the classical Einstein paradigm.  Note that in many literature instead of free jump term of free pass is used. From this point of view events of free jump and free pass are synonyms. \\
In his famous thesis (see \cite{Einstein56}), Einstein assumed that process of random motion of the particles is characterised by events of free jumps of the particles. Einstein proposed two main parameters which characterise free jumps : $\tau$ -the time interval within which the free jumps occur, and $\varphi(\Delta)$ - frequency of the occurrences of free jumps of the length $\Delta$. Einstein assumed that $\tau$ and the variance $ \displaystyle \sigma^2= {\int_{-\infty}^{\infty} {\Delta}^2 \varphi(\Delta)\ d \Delta} $ are to be constants which allowed him to reduce his sought experiment to the classical heat equation with constant coefficients where that exhibits the so called effect of infinite speed of the perturbation. Namely if at some moment $t_0$ of time and at any point $x_0$ concentration of particles $u(x,t) > 0$ then $u(x,t)>0$  for all $x$ and $t \geq t_{0}$. This property of $u(x,t)$ is very unrealistic. The question which we asked in this article is as follows: Can Einstein paradigm of free jumps be generalised in such way that one sees finite speed of propagation of the particles ? i.e. \textit{If $u(x,0)=0$ only at some point $x_0$ then $u(x_0,t)=0$   during time interval $[0,T]$ for some $T>0$.}\\ 
In this article we will show that if time interval of free jumps $\tau$ is inverse proportional to the concentration $u$ or/and $u$ is proportional to  $\sigma^{2}$ then, under some assumptions of the proportionality, our solution of degenerate Einstein equation in IBVP\eqref{ibvp} will exhibit the finite speed of propagation property. This property closely relate to so called Barenblatt solutions for degenerate porous medium equation (see \cite{Barenblatt-96}).
It is necessary to state that origin of the porous media equation is differ from sought experiment of the Einstein paradigm. Namely the first continuity equation hypothesised in the form:
\begin{equation}\label{cont-flux-eq}
L(\rho,\vec{J} )\triangleq\rho_t+\nabla\cdot\left(\rho \vec{J}\right)=0,
\end{equation}
where $\rho$ is density and $\vec{J}$ is flux in porous media that equals to the velocity of the flow - $\vec v$ (see \cite{Aronson-1}). Next it is assumed that $\rho$ is function of the pressure $p$ and can be approximated based on thermodynamic experiment. For example, $\rho$ can be approximated by $p^{\lambda}$ for some gasses by thermodynamic laws and  $ \displaystyle \vec{v}=-\frac{k}{\mu}\nabla p,$ where $k$ - permeability of porous media, and $\mu$ is viscosity subject to the experimental Darcy equation (see \cite{Muskat}). Combining these experimental relations in \eqref{cont-flux-eq} one can get the following divergent equation for scalar pressure function
\begin{align}\label{cont-p-eq}
 L (w)=
 (w)_t - \frac{k}{\mu} \ \frac{1}{\lambda+1} \ \Delta \left(w^{ \frac{\lambda+1}{\lambda}}\right) = 0.
\end{align}
where $w = p^\lambda$ and $ \lambda > 0 $. In the pioneering work by Barenblat-Kompaneetz-Zeldovich (see \cite{Evan}) it was shown that under specific initial and boundary conditions there exists self-similar solution of the equation \eqref{cont-p-eq}, which exhibits finite speed  of propagation. Evidently, the Einstein operator $L_{e}$(see \eqref{M-1}) is in non-divergent form and  based on sought experiment which allows to interpret the results on observable data in term of the length of free jumps, and then adjust parameters $\tau$ and $\varphi$ of the model. \\
In this paper we will follows Einstein approach but will use technique for divergent equation.
Article is organized as follows: In Section \ref{Gen-Ein-Para}, we consider generalization of classical Einstein model  of Brownian motion to $\mathbb{R}^{N}$,  when major parameters of the system - time interval $\tau$ of the free jumps with a length $\Delta$ (free passes) and the frequency $\varphi$ of the free jumps of the length $\Delta$ depend on the concentration of the number of particle per unite volume. We use the generic mass conservation law with absorption/reaction and basic stochastic principles to derive a partial differential inequality (PDI) under the Einstein's axioms in Section \ref{Derivation PDI}. By introducing local forces and non-locality processes we conclude the deterministic PDI  \eqref{M-1} which governs the dynamics of the generalized Einstein paradigm. Using Hypothesis \ref{tau}, In Section \ref{Nonlin-IBVP} we define $u$ as the weak viscous solution to the non-linear initial boundary value problem \eqref{ibvp} such that holds \eqref{eps_sol}. Introducing the assumption on functions $F,G,P$ in Section \ref{Assumption-sec}, we explicitly structures $H$ in \eqref{H-choice} and $F$ in \eqref{F-result},  in a way that $u$ conserves the finite speed of propagation, verifying $H$ and $F$ throughout those assumptions with necessary conditions on $\Lambda$. In section \ref{iterative ineq} we establish corroborating  Lemmas and introduce Ladyzhenskaya Iterative scheme which apply in Section \ref{localization} to prove the localization property of $u$, based on De-Georgi's construction, if the functional $Y_{0}(T)$ preserves the boundednes in \eqref{X_0-assump}. Next in Section \ref{Exmp-P-F} we show some counter examples of the functions $F$ and $P$ that attest all constrains on $ F, G ,H $. We present auxiliary results of Functional spaces in Section \ref{viscous-solution} to prove the uniform boundednesses regarding the classical solution $u_{\epsilon}$  which ensures its limit as in \eqref{eps_sol}. This Degenerate porous media equation is obtained different from Einstein conservation law which state that rate of growth of the density is proportional  to divergence of the flux.
   
 
\section {Generalized Einstein paradigm}\label{Gen-Ein-Para} 
Let  $x \in   \mathbb{R}^{N} $ and  $u(x, t)$ be the 
function which represent  the number of particles per unit volume at point $x$ and at time $t$. Consider a particle ($P$) of particular type suspended in the medium of interest. Denote $\mathbb{P}(\tau)$ to be the set of vectors with non-colliding jumps of $P$ corresponding to time interval $\tau$.We call  $\vec{\Delta} 
= \big(\Delta_1, \cdots ,\Delta_N\big)^{T}$ 
to be a '' vector of free jump of particles $P$'' if $\vec{\Delta} \in  \mathbb{P}(\tau)$.
We assume the following extension of the axioms in classical Einstein Brownian motion:

\begin{assumption} \  \normalfont
\begin{enumerate} 
\item
 Consider a particle ($P$) of particular type suspended in the medium of interest. Denote $\mathbb{P}(\tau)$ to be the set of vectors with non-colliding jumps of $P$ corresponding to time interval $\tau > 0$. We call  $\vec{\Delta} 
= \big(\Delta_1, \cdots ,\Delta_N\big)^{T}$ to be a  vector of free jump of particles $P$ if $\vec{\Delta} \in  \mathbb{P}(\tau)$.

\item
\label{constitive} Time interval of free jumps $\tau$, expected vector $\vec{\Delta}_e$  of  a free jump $\vec{\Delta}$ and probability density function of free jump $\varphi(\vec{\Delta})$  are the only parameters  which characterise process of free jumps. Note that in a view of the definition of the set $\mathbb{P}(\tau)$, if $\vec{\Delta} \notin \mathbb{P}(\tau)$ then  $\varphi(\vec{\Delta})=0$.
\item
 The key parameters $\tau$ and $\varphi$ in general can depend on the concentration of the particles and also space coordinate and time itself.
\end{enumerate}
\end{assumption}
\begin{axiom} {Whole universe axiom:}
 \begin{align}\label{uni-ax}
      \int_{\mathbb{P}(\tau)}\varphi(\vec{\Delta})d\vec{\Delta} = 1.
 \end{align}
\end{axiom}

\begin{assumption}{Symmetry of free jumps}
\begin{equation}\label{symmetry}
\varphi(\vec{\Delta}) = \varphi(-\vec{\Delta}).    
\end{equation}
Existence of second moments
\begin{equation}
    \int_{\mathbb{P(\tau)}}|\vec{\Delta}|^2\varphi(\vec{\Delta})d\vec{\Delta}<\infty
\end{equation}
\end{assumption}

Note that it follows from the preceding assumption that the expectation of $\vec{\Delta}$ equals zero and that there exists
a co-variance matrix $[\sigma_{ij}^{2}]$ of \textit{free jumps}:
 \begin{definition}{Standard Co-variance matrix of free jump}
 \begin{align}
\sigma_{ij}^2 \triangleq  \int_{\mathbb{P}(\tau)}{\Delta_{i}}{\Delta_{j}} \varphi(\vec{\Delta})d\vec{\Delta} \quad  \text{ where }  i,j = 1,2,\dots, N.
\label{var-ij}
\end{align}
\end{definition}
Evidently  $\sigma_{ij}(x,t)$ depend on space $x$ and time $t$. We postulate generalized Einsteins Axiom for the number of particles found at time $t+\tau$ in the control volume $dv$  contained point $x$ by :
\begin{Axiom}\begin{equation}
u(x, t+\tau) \cdot  dv =  
\int_{\mathbb{P}(\tau)} u(x+ \vec{\Delta}, t)\varphi(\vec{\Delta}) d \vec{\Delta} \cdot dv
+
\tau\cdot\int_{\mathbb{P}(\tau)} M(x+ \vec{\Delta}, t)\varphi(\vec{\Delta}) d \vec{\Delta} \cdot
dv.\label{Einstein_conserv_eq}
\end{equation}
\end{Axiom}
Mass conservation law \eqref{Einstein_conserv_eq} intuitively is easy to interpret, and it says that at any given point in space $x$ at time $t+\tau$ we will observe density of all particles with free jumps from the point $x$ at time $t$,  "$+$ density" of  particles which "produced" and "$-$ density" of  particles which "consumed" during time interval $[t,t+\tau].$  For comparison see \cite{Einstein56} (pages 14) first formula with integral. We modeled this by adding the term $ M(\cdot)$ which model process of rate of absorption/consumption during time interval $[t,t+\tau]$ due to particles interaction  along of free jumps and therefore it set to be non-positive.  
\begin{remark}
Einstein definition of density of particles  differs from the fundamental definition of the density of fluid, rather mean the concentration.
\end{remark}
\section{Derivation of Partial differential Inequality} \label{Derivation PDI}
In this section we will derive partial differential inequality from \eqref{Einstein_conserv_eq} whose solution will exhibit the feature of finite speed of propagation.  
Let $\zeta = (\zeta_1,\zeta_2,\cdots,\zeta_N)$ multi-index and  $ x^{\zeta} \triangleq x_{1}^{\zeta_1} \cdot x_{2}^{\zeta_2} \cdot \dots \cdot x_{N}^{\zeta_N}. $ Assume that $u(x,t)\in C_{x,t}^{2,1} $. By Taylor's Expansion (see \cite{kon}), and using \eqref{uni-ax} - \eqref{var-ij} we get
\begin{align}
\int_{\mathbb{P}(\tau)}
u(x +\vec{\Delta},t)  \varphi(\vec{\Delta})d\vec{\Delta}
 \text{ = } 
\text{ } u(x,t) 
 \text{ + }
\frac{1}{2}\sum_{i,j=1} \sigma_{ij}^{2}u_{x_{i}x_{i}}{(x,t)}
\text{ + }
R_{\zeta}\label{Taylor-eq}
\end{align} 
where 
\begin{align}
R_{\zeta} \triangleq {\int_{\mathbb{P}(\tau) }
 \sum_{|\zeta| = 2}   
H_{\zeta}(x,\vec{\Delta},t)(\vec{\Delta})^{\zeta}
\varphi(\vec{\Delta})d\vec{\Delta}} \label{R-term}
\end{align}
with $H_{\zeta}$ locally bounded and
$\lim_{\vec{\Delta}\rightarrow 0} H_{\zeta}(x ,\vec{\Delta},t) = 0.
$ 
Using above in $\eqref{Einstein_conserv_eq}$ we get
\begin{equation}
u(x , t + \tau) - u(x,t) = 
 \frac{1}{2}\sum_{i,j=1}^N  \sigma_{ij}{u}_{x_{i} x_{j}} {(x,t)} + R_{\zeta} 
 + 
\tau\int_{\mathbb{P}(\tau)} M(x+ \vec{\Delta}, t)\varphi(\vec{\Delta}) d \vec{\Delta} .
 \label{post-Taylor}
\end{equation}
Observe that LHS and RHS in the equation \eqref{post-Taylor} are defined at different times. In order to eliminate this ambiguity and derive the equation at the same point we will assume that $u(x,t) \in C^{2,1}_{x,t}$. By Carathéodory's criterion    $\exists$ function $\psi^{t}$ such that 
\begin{align}\label{Caratheodory-1}\small
& u(x, t+ \tau)  - u (x, t) =
\tau \psi^{t} (x,t,\tau) 
\end{align}
where $  \lim_ {\nu \rightarrow 0} \psi^{t}(x,t,\nu) =u_{t}(x,t)$. Using  \eqref{post-Taylor} in \eqref{post-Taylor} we get 
\begin{equation}
\tau u_{t}(x,t) \approx
 \frac{1}{2}\sum_{i,j=1}^N  \sigma_{ij}{u}_{x_{i} x_{j}} {(x,t)} + R_{\zeta}
 + 
\tau\int_{\mathbb{P}(\tau)} M(x+ \vec{\Delta}, t)\varphi(\vec{\Delta}) d \vec{\Delta} .
 \label{post-Taylor}
\end{equation}
\begin{assumption}
 Mathematically for our forthcoming iterative procedure we will assume that
 \begin{align}
 R_{\zeta}+ \tau \int_{\mathbb{P}(\tau)} M(x+ \vec{\Delta}, t)\varphi(\vec{\Delta}) d \vec{\Delta}  \leq 0 . \label{Negative-pheno}
 \end{align}
\end{assumption}
Using   \eqref{Negative-pheno} in \eqref{post-Taylor} and approximating the first order terms, it follows that the function $u(x,t)$ can be estimated by the following Partial differential inequality:
\begin{align}
   L_{e} u =
   \tau{u_t}-
    \sum_{i,j=1}^N &  \sigma_{ij} u_{x_{i}x_{j}} \leq 0.
\label{M-1} 
\end{align}
Observe that the partial differential inequality \eqref{M-1} is deterministic, while the stochastic nature is modeled by functions    $\tau$ and $\varphi$ (the latter appears in \eqref{M-1} in form of its co-variance matrix [$\sigma_{ij}$]), which  are key characteristics of the process dynamics. In general, they can be functions of spatial and time variables  $x$ and $t$, concentration $u$ and its gradient $\nabla u$. In this report we present $\tau$ and $\sigma$ satisfy the following Hypothesis.
\begin{hypotheses}\label{tau}

In addition to \eqref{uni-ax} and 
\eqref{symmetry}, let 
\begin{equation}\label{def-P-compos}
    c_1P(u)|\xi|^2 \leq \sum\limits_{ij=1}^N \frac1\tau \sigma_{ij}\xi_i\xi_j\leq c_2P(u)|\xi|^2, \quad \xi\in\mathbb{R}^N
\end{equation}
for some $0<c_1<c_2$ and $ P\in C([0,\infty))$, $P(0) = 0$, $ 0 < P(s) < C < \infty$ for $s>0$.
\end{hypotheses}
\begin{remark}
Note that in \eqref{def-P-compos} $P(u)$ is a composite parameter of stochastic processes of free jump, which characterise relative co-variance of free jump with respect to time interval of latter, which in our scenario degenerate at $0$. Which mean that dispersion growth slower   than time interval of this dispersion, as the concentration vanishes.   
\end{remark}
\section{Non-linear Degenerate IBVP} \label{Nonlin-IBVP}
Let $\sigma > 0$. For simplicity of the argument  we will restrict the exposition to the case $\sigma_{ij} = \sigma \delta_{ij}$ for $i,j=1,\dots,N$ so that $P=\dfrac{\sigma}{\tau}$. Hence \eqref{M-1} takes form
\begin{align}
    u_t \leq P(u) \Delta u. \label{M-2}
\end{align}
\begin{definition}\label{G-def}\
\begin{enumerate}
    \item Let  $h$ be a positive function such that $h\in C((0,\infty))$  and integrable at zero w.r.t. $u$. Then we define \begin{equation}\label{H-F-def}
  H(u) \triangleq\int_0^u h(s) ds 
\end{equation}
    \item Let $F \triangleq hP$ and $h$,$P$ be such that $F(0)=0$, and $F$ is differentiable on $(0,\infty)$ with a locally bounded derivative $F^{'}>0$. 
    \item Let $G$ be s.t.  $\sqrt{F^{'}(u)}=G^{'}(u)$ and $G(0)=0$. Then $\displaystyle G(s)\triangleq \int_{0}^{s} \sqrt{F^{'}(s)} \ ds. $ \label{F-def}
\end{enumerate}
\end{definition}
\begin{remark}
Note that 
\begin{equation}\label{G-F-2} 
0\leq G(u)=\int_0^u G^{'}(s)ds \leq \sqrt{u\int_0^u F^{'}(s)ds} = \sqrt{uF(u)}.
\end{equation}
In above, all functions $F,G$ and $H$ are increasing w.r.t. $s$. 
\end{remark}
Further restriction on $P$ and of $h$ will be presented later.
We multiply \eqref{M-2} by $h(u)$ to get.
\begin{align}
 \dfrac{\partial}{\partial t} H(u) -   F(u)\Delta u \quad \mbox{on } \Omega\times(0,T)\ , \label{L-OP}
\end{align}
for some $T>0$. Here $u$ is positive measurable such that $u(\cdot,t) \to u(\cdot,0)$ as $t\to 0$ in (locally) measure, and
\begin{equation}\label{space-for-solution}
\begin{aligned}
u\in L_{loc}^{\infty}\left(\Omega\times [0,T]\right),\quad
\nabla u\in L^{2}_{loc}\left( \Omega \times (0,T]\right),
\quad u_t\in  L^{1}_{loc}\left( \Omega \times (0,T]\right)
.
\end{aligned}
\end{equation}
So $F(u)\Delta u$ is understood in the weak sense: for all
$\theta\in Lip(\Omega)$ with a compact support,
\begin{equation}\label{weak Laplace}
\begin{aligned}
-\int \theta F(u)\Delta u dx & = \int \left[F(u)(\nabla u)(\nabla \theta) + F'(u) |\nabla u|^2\theta\right]dx\\
& = \int \left[F(u)(\nabla u)(\nabla \theta) +  \left|\nabla G(u)\right|^2\theta\right]dx.
\end{aligned}
\end{equation}
Then under the Hypotheses \ref{tau},  the differential inequality \eqref{M-1}  we defined $u(x,t)$ as a non-negative  solution of the following differential inequality and initial condition which we for convenience will call as IBVP
\begin{align}
\begin{cases} 
\ \dfrac{\partial}{\partial t} H(u) -   F(u)\Delta u   \ \leq 0 \  & \text{ in }  \Omega\times (0,T)
\label{ibvp} \\
\hspace{2.3 cm} u(x,0)   \  = 0    &\mbox{ in } \Omega^{'} \Subset \Omega   \\
\hspace{2.3 cm} \ u(x,t)  =\hspace{0.08 cm} \phi(x,t)         & \mbox{ on } \ {\partial \Omega \times (0,T)}.
 \end{cases} 
 \end{align}
where $\phi\geq 0$ and $\displaystyle \phi(x,t) = 0$ when $t \rightarrow 0$. The condition on the boundary $\partial\Omega\times(0,T)$ is not essential to prove that for every ball $B_{R}(x_{0}) \Subset \Omega^{'}$ and $c<1$, $\exists$ $T'=T'(x_0,c)\in (0,T]$ such that $u(x,t)=0$ for all $(x,t) \in B_{cR}(x_0)\times[0,T']$ in Theorem \ref{Main-T}. We also assume that $u(x,0)$ is continuous on $\Omega$. Observe that $u$ in above IBVP is degenerates when $u \rightarrow 0$. Therefore its solution $u(x,t) \not \in \mathcal{C}^{2,1}_{x,t}(\Omega\times (0,T])$. 
 Our result is qualitative and does not address the existence of the solutions, but the obtained property of the solution will be applicable for the weak viscous solution which one can define as follows. 
 Let us define the regularized IBVP$_{\epsilon}$  for some $u_{\epsilon} \in \mathcal{C}^{2,1}_{x,t}(\Omega\times (0,T])$ as follows:
\begin{align} 
\begin{cases} 
\ \dfrac{\partial}{\partial t} H(u_{\epsilon}) -   (F(u_{\epsilon})+\epsilon)\Delta u_{\epsilon}   \ \leq 0 \  & \text{ in }  \Omega\times (0,T)
\label{ibvp-ep} \\
\hspace{3.7 cm} u_{\epsilon}(x,0)   \  = \epsilon    &\mbox{ in } \Omega^{'} \Subset \Omega   \\
\hspace{3.7 cm} \ u_{\epsilon}(x,t)  =\hspace{0.08 cm} \phi(x,t) + {\epsilon}         & \mbox{ on } \ {\partial \Omega \times (0,T)}.
 \end{cases} 
 \end{align}
Assume that $u_{\epsilon}(x,0)$ is continuous in $\Omega$. In this article, the constants in all estimates for $u_{\epsilon}$  do not depend on $\epsilon$. This observation allow us to pass to the limit in the final estimates, and conclude the localization property for the limiting function:
\begin{equation}\label{eps_sol}
u(x,t)=\lim_{\epsilon\to 0}u_{\epsilon}(x,t),
\end{equation}   
which is considered as a weak passage to the limit (see \cite{CIL92}). The obtained function $u(x,t)$ is called a weak viscous solution of the IBVP \eqref{ibvp} with first differential inequality replaced by by the differential equation $Lu=0$ , which will exhibit localisation property. Further details on weak viscous solution of IBVP with homogeneous boundary function will be presented in Section \ref{viscous-solution}.  
\section{Assumptions on Functions and interpretation}\label{Assumption-sec}
In this section we state main properties and additional assumption, which we imposed on the functions $H$ and $F$ besides monotonously, we will require  for some finite $M>0$ following assumption holds
\begin{assumption}  \label{Assumps}
\begin{align}
   \exists\ C_{1} > 0 \mbox{ such that }  F(s)\leq C_{1}G^{'}(s)G(s),~ s\in[0,M]. \ \text{ Here } \  G(s)= \int_{0}^{s} \sqrt{F^{'}(s)} \ ds.  \label{A-1}\\
 \displaystyle \exists \   \lambda \in (0,2), C_{2}>0  \mbox{ such that }\left(\sqrt{sF(s)}\right)^{\lambda}\leq {C_{2}H(s)}, ~ s\in[0,M],  \label{A-2} 
\end{align}
\end{assumption}
We will choose $H$ and impose restrictions on $P$ such that 
Assumption \ref{Assumps} hold.
\begin{proposition} \label{P-1}
Let $F$ and $H$ hold \eqref{A-2}. Then
\begin{align}\label{H-prop}
    {H(s)\geq
\left( H^{-\Lambda}(M) + \Lambda C_{2}^{1+\Lambda}\int_{s}^{M} \frac{1}{\tau P(\tau)} d\tau \right)^{-\frac{1}{\Lambda}}},
\end{align}
with $\Lambda + 1= \dfrac{2}{\lambda}$ and $\infty > \Lambda > 0$.
\end{proposition}
\begin{proof}
We raise \eqref{G-F-2} to power $\lambda$ and use \eqref{A-2} to get, 
\begin{align}
  G^{\lambda}(s)\leq \left(\sqrt{sF(s)}\right)^{\lambda}  \leq {C_{2}H(s)} \implies
  h(s)P(s) = F(s) \leq C_{2}^{\frac{2}{\lambda}} \dfrac{H^{\frac{2}{\lambda}}(s)}{s}\label{before-beta}
\end{align}
Now we rewrite right hand side of  \eqref{before-beta} as follows:
\begin{align}
\frac{h(s)}{H^{\Lambda+1}(s)} 
& \leq \frac{C_{2}^{1+\Lambda}}{sP(s)}  \label{before-int}
\end{align}
Integrate  \eqref{before-int} over $(s,M)$ we get \ref{H-prop}.
\end{proof}
The preceding proposition implies the following choice of function $H$.
\begin{definition}\label{PIH}
Let $P$ in Hypotheses\eqref{tau} be such that \ 
$ \displaystyle
0 < \int_M^\infty\frac{ds}{sP(s)} <\infty
$, for some finite $M$.
Then for $s\in[0,M]$, let
\begin{align}
     I(s) & \triangleq \int_{s}^{\infty}\frac{d\sigma}{\sigma P(\sigma)} \ ; ~ M > s > 0 ,\label{I-s}
\end{align}
\begin{align}\label{H-choice}
\boxed{
 H(s) \triangleq 
    \left[\Lambda I(s)\right]^{-\frac{1}{\Lambda}} =
    \left(\Lambda \int_{s}^{\infty} \frac{1}{\tau P(\tau)} d\tau\right)^{-\frac{1}{\Lambda}},  
    ~ s>0.
}\end{align}
Function $P(s)$ and consequently $H(s)$ define only on bounded interval, But in order notation to be simple we extended $I(s)$ on whole axis in such order that $I(s)>constant>0$ on the interval $[M,\infty)$.
\end{definition}
\begin{remark} \label{new-def}
\normalfont
Note that, $H$ in \eqref{H-choice} has $H(0) = 0$ since $P(0) = 0$. Moreover, substitute \eqref{H-choice} in \eqref{H-F-def} to get
\begin{align}
    h(s) = \frac{1}{sP(s)}\left[  \Lambda I(s) \right]^{-\frac{1}{\Lambda}-1}
    = \frac{1}{sP(s)} H^{(\Lambda+1)}(s) 
    \ , ~ s>0; \label{h-ret}
\end{align}
Using \eqref{h-ret} for (ii) in Definition \ref{G-def} we get
\begin{align}
\boxed{ 
    F(s)=h(s)P(s)=\frac{1}{s} H^{(\Lambda+1)}(s)=\left(\Lambda s^{\frac{\Lambda}{1+\Lambda}}I(s) \right)^{-\frac{1}{\Lambda}-1}.
    \label{F-result}
}
\end{align}
Consequently,
\[
\left[sF(s)\right]^{\frac\lambda2} = 
H^{(\Lambda+1)\frac\lambda2}(s)=H(s),
\]
since $(\Lambda+1)\frac\lambda2=1$. Thus, \eqref{A-2} holds with $C_2=1$. 
\end{remark}
The next proposition imposes restrictions on $P$ such that $F$ as in \ref{F-result},
is non-negatively increasing on $[0,M)$.
\begin{proposition}\label{P-2}
Let $P$ be in Definition \ref{PIH} and $F$ be in \eqref{F-result}. Let 
\begin{align}
A & \triangleq  \sup_{0 < s < M} P(s)I(s) < \infty \label{A}\\
a & \triangleq  \limsup_{s\rightarrow 0} P(s)I(s) < A. \label{a}
\end{align}
If $\frac{\Lambda+1}{\Lambda} > A $ then $F'(s) > 0 $ on $(0,M)$,
and 
if $\frac{\Lambda+1}{\Lambda} > a$ then $\lim_{s \rightarrow 0}F(s) = 0$.
\end{proposition}
\begin{proof}
First we prove that function $F$ is increasing. It suffices to show the map
\begin{equation}s \mapsto \displaystyle{ s^{\frac{\Lambda}{1+\Lambda}}I(s)}
\end{equation}
decreases on $(0,M)$.
Note that
\begin{align}
\displaystyle
\frac{d}{ds}\left( s^{\frac{\Lambda}{1+\Lambda}}I(s)\right)
=
\left( \frac{\Lambda}{\Lambda+1} \frac{1}{P(s)}\right)\left( P(s)I(s) -\frac{\Lambda+1}{\Lambda}\right)s^{-\frac{1}{\Lambda+1}}. \label{comp-deri}
\end{align}
Thus, provided  $\frac{\Lambda +1}{\Lambda} > A$, one has $F^{'}(s) > 0 $ for $s\in(0,M)$. Next we show that
$ \displaystyle
 s^{\frac{\Lambda}{1+\Lambda}}I(s) \rightarrow \infty \text{ as } s\rightarrow 0$, which implies $\lim_{s \rightarrow 0}F(s) = 0$.
By \eqref{a}, for every $\epsilon > 0 $ there exists $s_{\epsilon} \in (0,M)$ such that $ I(s)P(s) < a+\epsilon$ for $s\in(0,s_\epsilon)$, which yields the following chain of inequalities:
\begin{align*}
 I(s) < (a + \epsilon)\frac{1}{P(s)}  &  = -(a+\epsilon)s I^{'}(s)\\
    \frac{d}{ds} \ln I(s)  & < -\frac{1}{a+\epsilon} \cdot \frac{1}{s}\\
I(s) & >  I(s_{\epsilon})\left(\frac{s_{\epsilon}}{s}\right)^{\frac{1}{a+\epsilon}}  \\
s^{\frac{\Lambda}{\Lambda+1}}I(s) & \geq
I(s_{\epsilon})s_{\epsilon}^{\frac{1}{a+\epsilon}}\times s^{\frac{\Lambda}{\Lambda+1}-\frac{1}{a+\epsilon}}
\ ; \ 0 < s < s_\epsilon.
\end{align*}
Provided $\frac{\Lambda +1}{\Lambda} > a$, one can choose $\epsilon$ such that $\frac{\Lambda +1}{\Lambda} > a + \epsilon$, hence
$\frac{\Lambda}{\Lambda+1}-\frac{1}{a+\epsilon}<0.$ \[
 \therefore \ \  s^{\frac{\Lambda}{\Lambda+1}}I(s) \to \infty \text{ as } s\to 0. \qedhere\]
\end{proof}
\begin{remark} \label{A-a-beta}
If $A$ and $a$ exist then indeed \  $\displaystyle \frac{\Lambda+1}{\Lambda} > A > a $ in Proposition \ref{P-2}. However, $A<\infty$  only if $a<\infty$ since
function $P(s)I(s)\in C(0,M]$.
\end{remark}
\begin{definition}[Equivalence of functions and almost monotone functions] \
\begin{enumerate}
    \item 
Functions $f$ and $g$ of a set $E$ are {\it equivalent} and write
 $f \asymp  g$ on $E$ if there exist a constant $c\geq1$ such that $ \displaystyle \frac{1}{c} \  g(x) \leq f(x) \leq c \ g(x)$ for all $x\in E$. 
 \item
A function $f$ on an interval $I\subset \mathbb{R}$ is called almost decreasing (almost increasing) if it is equivalent to a non-increasing (non-decreasing) function of an interval.
Equivalently, there exists a constant $c>0$ such that 
\[
f(t)\geq cf(s), \quad t,s\in I,~ t<s.
\]
\end{enumerate}
 \end{definition}
 In the following proposition we give assumptions providing for the inequality \eqref{A-1} in Assumption \ref{Assumps}. 
\begin{proposition} \label{P-3}\normalfont
Let $I$ be as in Definition \ref{PIH},  $P$  in  Proposition \ref{P-2}, and $F$   in \eqref{F-result} with $\Lambda$ as in Remark \ref{A-a-beta}. Then \eqref{A-1} holds provided the existence of $\mu>0$
such that $s\mapsto P(s)I^\mu(s)$ is an almost decreasing function.
\end{proposition}
\begin{proof} By the direct computation, it follows from \eqref{F-result} that 
\begin{align}
  F(s) \asymp  \  &      \ s^{-1}I^{-\frac{1}{\Lambda}-1}(s) \label{R-1} \\
  F^{'}(s) \asymp  \  &  \  \frac{F(s)}{s}[P(s)I(s)]^{-1} \label{R-2}.
\end{align}
Using  \eqref{R-1},\eqref{R-2}  and (iii) in Definition \eqref{G-def} as following:
\begin{align}
\frac{F(s)}{G(s)G^{'}(s)} & = 
    \frac{\displaystyle F(s)}{\displaystyle \left(\int_{0}^{s} \sqrt{F^{'}(t) \ dt} \right)\displaystyle  \left( \sqrt{F^{'}(s)}\right) }\\
    & \asymp \frac{\displaystyle F(s)}{\displaystyle \left( \int_{0}^{s} \sqrt{\displaystyle \frac{F(t)}{t}[P(t)I(t)]^{-1}} \ dt\right)\left( \sqrt{\displaystyle \frac{F(s)}{s}[P(s)I(s)]^{-1}}\right) }\\
     & =
    \frac{\displaystyle [I(s)]^{-\frac{1}{2\Lambda}-\frac{1}{2}} [P(s)I(s)]^{\frac{1}{2}}}{\displaystyle \int_{0}^{s} t^{-1}[I(t)]^{-\frac{1}{2\Lambda}-\frac{1}{2}} [P(t)I(t)]^{-\frac{1}{2}} \ dt}\\
     & =
    \frac{\displaystyle [I(s)]^{-\frac{1}{2\Lambda} - \frac\mu2} [P(s)I^{\mu}(s)]^{\frac{1}{2}}}{\displaystyle \int_{0}^{s} t^{-1}[I(t)]^{-\frac{1}{2\Lambda}-1} [P(t)]^{-\frac{1}{2}} \ dt} \label{after 87} 
\end{align}
By the Cauchy mean value theorem (aka extended mean value theorem),
there exist $t\in(0,s)$ such that
\[
\frac{\displaystyle [I(s)]^{-\frac{1}{2\Lambda} - \frac\mu2}}{\displaystyle \int_{0}^{s} t^{-1}[I(t)]^{-\frac{1}{2\Lambda}-1} [P(t)]^{-\frac{1}{2}} \ dt}
= \left(\frac{1}{2\Lambda} + \frac\mu2\right)
\frac{\displaystyle [I(t)]^{-\frac{1}{2\Lambda} - 1 - \frac\mu2}[tP(t)]^{-1}}{\displaystyle t^{-1}[I(t)]^{-\frac{1}{2\Lambda}-1} [P(t)]^{-\frac{1}{2}}}
\asymp [P(t)I^{\mu}(t)]^{-\frac12}.
\]
Note that $\displaystyle I^{'}(s) =  \  [sP(s)]^{-1} $.
Since $t\mapsto P(t)I^{\mu}(t)$ is almost decreasing one has
\begin{equation}
 \frac{F(s)}{G(s)G^{'}(s)} \asymp 
 \left(\frac{P(s)I^{\mu}(s)}{P(t)I^{\mu}(t)}\right)^{\frac12} 
\leq C \ ; \  0<t<s . \qedhere
\end{equation} 
\end{proof}
The sufficient condition for the existence of $\mu>0$ for the almost decreasing function  $t\mapsto P(t)I^{\mu}(t)$ is as follows
\begin{remark}\label{I-B-prop}
Assume that there exists $\tilde P\in C^1(0,\infty)$ such that
 $P\asymp \tilde P$ on $(0,M)$, such that
 \begin{equation}\label{R-4}
     \limsup_{s \to 0}   s\tilde I(s) \tilde P^{'}(s) <\infty,
 \end{equation}
 with 
 \[
 \tilde I(s) = \int_s^\infty \frac{dt}{t\tilde P(t)}.
 \]
 
 Indeed, \eqref{R-4} implies that 
 $s\mapsto s\tilde I(s) \tilde P^{'}(s)$ is bounded on 
 $(0,M)$. Let
 \[
 B=\sup\limits_{0<s<M}s\tilde I(s) \tilde P^{'}(s).
 \]
 Fix  $\mu\geq B$. Then
 the function $Q(s) = \tilde P(s)\tilde I^{\mu}(s)$ is non-increasing since
 \[
 Q'(s) = \tilde P^{'}(s)\tilde I^{\mu}(s) - \mu s^{-1}\tilde I^{\mu-1}(s) = s^{-1}\tilde I^{\mu-1}(s)
 \left[s\tilde I(s) \tilde P^{'}(s) - \mu\right]\leq 0.
 \]
 Finally, note that $P(s) I^{\mu}(s)\asymp\tilde P(s)\tilde I^{\mu}(s)$.
\end{remark} 
\begin{section}{Auxiliary integral estimates  }\label{iterative ineq}
We will start with Auxiliary generic estimates for function $u,$ and cutoff functions $\theta$ w.r.t. non-linear functions $F,$ and $G$.   
\begin{lemma} \label{lemma-1}
Let $u$ and $\nabla u$ be measurable functions, and let $\theta \in Lip_{c}(\Omega)$. Let $F, G\in C^1(0,M)\cup C[0,M]$ satisfy \eqref{A-1}. Then
\begin{equation}
\nabla u\cdot \nabla\left( \theta^{2} F(u)\right) \geq 
\frac{1}{2}| \nabla (\theta G(u))|^{2} - (2C_{1}^{2}+1) G^{2}(u)\vert\nabla \theta \vert^{2}.
    \label{lemma-1-R}
\end{equation}
\end{lemma}
\begin{proof}
\begin{align}
\nabla u\cdot \nabla  \left( \theta^{2}F(u)\right)
& = \theta^{2}F^{'}(u) \vert\nabla u \vert^{2} + 2F(u)\nabla u\cdot \theta \nabla \theta \nonumber\\
& = \theta^{2}F^{'}(u)\vert\nabla u \vert ^{2} + 2\dfrac{F(u)}{G^{'}(u)}
G^{'}(u)\nabla u\cdot \theta\nabla \theta \quad  ; \ G^{'}(u) > 0
\end{align}
Using \eqref{A-1} in right-hand side of above yields
\begin{align}
& \nabla u\cdot \nabla\left( \theta^{2}F(u)\right) \nonumber\\
&= \theta^{2}\vert G^{'}(u)\vert^{2} \vert\nabla u\vert^2 + 2\frac{F(u)}{G^{'}(u)}  \left( G^{'}(u) \nabla u \cdot \theta  + G(u) \nabla \theta \right) \cdot \nabla \theta -2\frac{F(u)}{G^{'}(u)}G(u)|\nabla \theta|^{2} \nonumber \\
& = |\nabla(\theta G(u))-G(u)\nabla \theta|^2 
+ 2 \dfrac{F(u)}{G^{'}(u)} \nabla(\theta G(u)) \cdot \nabla \theta -2 \dfrac{F(u)}{G^{'}(u)} G(u)\vert\nabla\theta\vert^{2} \nonumber\\
&=|\nabla(\theta G(u))|^{2}
+2\left( \dfrac{F(u)}{G^{'}(u)}- G(u) \right)\cdot \nabla \theta \cdot\nabla(\theta G(u))\nonumber - 2\left( \dfrac{F(u)}{G^{'}(u)} - G(u) \right)|\nabla \theta |^{2}G(u) \nonumber
\end{align}
\begin{multline}
\nabla u\cdot \nabla\left( \theta^{2}F(u)\right) \\
\geq |\nabla(\theta G(u))|^{2}
- \left \vert 2\left( \dfrac{F(u)}{G^{'}(u)}- G(u) \right)\cdot \nabla \theta \cdot\nabla(\theta G(u)) \right \vert
-2\left( \dfrac{F(u)}{G^{'}(u)} - G(u) \right)|\nabla \theta |^{2}G(u).
\label{before-cauchy}
\end{multline}
Note that by Cauchy's Inequality,
\begin{align}
2\left( \dfrac{F(u)}{G^{'}(u)}- G(u) \right) \nabla \theta \cdot\nabla(\theta G(u))
\leq
2\left( \dfrac{F(u)}{G^{'}(u)}- G(u) \right)^{2} |\nabla \theta |^{2} + \dfrac{1}{2}|\nabla(\theta G(u))|^{2}. \nonumber
\end{align}
Then inequality \eqref{before-cauchy} becomes 
\begin{align}
\nabla u\cdot \nabla\left( \theta^{2}F(u)\right) & \geq \dfrac{1}{2}|\nabla(\theta G(u))|^{2}
- 2\left( \dfrac{F(u)}{G^{'}(u)}- G(u) \right)^{2}\cdot |\nabla \theta |^{2}
-\left(2 \dfrac{F(u)}{G^{'}(u)} - G(u) \right)|\nabla \theta |^{2}G(u) \nonumber \\
& = \dfrac{1}{2}|\nabla(\theta G(u))|^{2}- \left[  2 \left(\dfrac{F(u)}{G^{'}(u)}\right)^{2} -2\dfrac{F(u)}{G^{'}(u)} G(u) + G^{2}(u) \right] |\nabla \theta|^{2}\\
 &\geq
\frac{1}{2}| \nabla (\theta G(u))|^{2} - (2C_{1}^{2}+1) G^{2}(u)|\nabla \theta|^{2}. \qedhere
\end{align}
\end{proof}
\begin{lemma}\label{lemma-2}
Assume \eqref{A-2}. Let $u$ be a measurable function of $\Omega$, $0 < u < M$, $\nabla u\in L^2_{loc}(\Omega \times (0,T)) $ and $\nabla G(u) \in L^2_{loc}(\Omega \times [0,T])$. Let $\displaystyle j=\frac{2}{N-2} > 0$ for Gagliardo–Nirenberg–Sobolev inequality \begin{align}
\vert\vert\psi\vert\vert_{L^{2+2j}}^{2}
\leq
S \vert\vert \nabla \psi \vert\vert_{L^{2}}^{2}.\label{Gil-Nir}
\end{align}
Let $K\subset \Omega$ be compact and $ \theta_{n} \in Lip_{c} (\Omega)$ such that $ \theta_{n} = 1$ on $K$.
Then,
\begin{align}
\int_{0}^{t} \int_{K} G^{2}(u) dxdt   
\leq c_1
t^{1-(1+j)k}
\left[  \sup_{0\leq \tau \leq T} \int_{\Omega} \theta_{n}^{2} H (u(\tau)) dx
+
 \int_{0}^{t}\int_{\Omega} |\nabla(\theta_{n} u)|^2 dx d\tau
    \right]^{1+jk}\label{lemma-2-R}
\end{align}
where $ \displaystyle k =\frac{2-\lambda}{2+2j-\lambda}$, $c_1=C_2^{1-k}S^{k(1+j)}$ with $C_2$ as in \eqref{A-2}.
\end{lemma}
\begin{proof}
Note that $\displaystyle \lambda = [2 - (2+2j)k] / [1-k]$. By \eqref{A-2} recall that $\left(\sqrt{sF(s)}\right)^{\lambda}\leq {C_{2}H(s)}$. Then one has 
\begin{align}
G^{2} (u) \leq C_{2}^{1-k} G^{2(1+j)k}(u)H^{1-k}(u). \label{G-H-k}
\end{align}
Integrate both side of \eqref{G-H-k} over $ K \times (0,t)$.
\begin{align*}
\int_{0}^{t} \int_{K} G^{2}(u) dxdt  
\leq & 
\ C_2^{1-k} \int_{0}^{t}\int_{K}G^{2(1+j)k}(u)H^{1-k}(u) dxd\tau \nonumber \\
\leq & 
\ C_2^{1-k} \int_{0}^{t}\int_{K} \left(|\theta_{n}G(u)|^{2(1+j)}\right)^{k}
\left( \theta_{n}^{2}H(u)\right)^{1-k} dx d\tau\\
\leq &
\ C_2^{1-k} \int_{0}^{t} 
    \left[\int_{\Omega} |\theta_{n}G(u)|^{2(1+j)} dx\right ]^{k}
\left[ \int_{\Omega}\theta_{n}^{2}H(u) dx\right]^{1-k}   d\tau\\
\leq&
\ C_2^{1-k}S^{k(1+j)}
 \int_{0}^{t} 
    \left[\int_{\Omega} |\nabla(\theta_{n}G(u))|^{2} dx\right]^{(1+j)k}d\tau \left[\sup_{0\leq \tau \leq T}\int_{\Omega}\theta_{n}^{2}H(u) dx\right]^{1-k} \nonumber 
\end{align*}
by \eqref{Gil-Nir}. We apply the Holder inequality for time integral and then using the estimate $ x^{v}y^{w}\leq (x+y)^{v+w} \ ; \ x,y,v,w > 0 $ to get the following:
\begin{align}
\int_{0}^{t} \int_{K} G^{2}(u) dxdt \nonumber  & \leq    C_2^{1-k}S^{k(1+j)} \left[\sup_{0\leq \tau \leq T}\int_{\Omega}\theta_{n}^{2}H(u) dx\right]^{1-k} t^{1-k(1+j)}
 \left[\int_{0}^{t} \int_{\Omega} |\nabla(\theta_{n}G(u))|^{2} dx  d\tau\right]^{(1+j)k} \nonumber \\
 & \leq
 C_2^{1-k}S^{k(1+j)} t^{1-k(1+j)} \left[\sup_{0\leq \tau \leq T}\int_{\Omega}\theta_{n}^{2}H(u) dx + \int_{0}^{t} \int_{\Omega} |\nabla(\theta_{n}G(u))|^{2} dx  d\tau\right]^{1+jk} \nonumber 
\end{align}
\end{proof}

\section{Localization of the solution of degenerate Einstein Equation}\label{localization}
\begin{definition}\label{R-teta}
Let $ R > 0$ and $ b > 2$. Let $R_{n}$ be a sequence such that
\begin{align}
    R_{0} = & R,\\
    R_{n} = & R_{n-1}-Rb^{-n}
          =  R\left(\frac{ b-2+ b^{-n}}{b-1} \right) ;n=1,2, \dots,\\
    R_{\infty} = & \lim\limits_{n\to\infty}R_n = R\left(\frac{b-2}{b-1}\right). 
\end{align}
\noindent Let $\theta_{n}(x) \in Lip(\Omega)$ such that 
\begin{align}
&\theta_{n}(x)
\triangleq
   \min \left\{\dfrac{\left( R_{n}-\vert\vert x-x_{0}\vert\vert\right)_{+}}{R_{n}-R_{n+1}}, 1\right\}
 = 
\begin{cases} 
 0 \ ; \ x \notin B_{R_{n}}(x_{0})\\
 1 \ ; \ x \in B_{R_{n+1}}(x_{0}) 
\end{cases} 
\end{align}
for $x_{0}\in {\Omega}^{'} \Subset \Omega$
and $B_{n}\triangleq  B_{R_{n}}(x_{0})$. 
\end{definition}
Then one can show that
\begin{align}
\vert\vert \nabla \theta_{n}(x) \vert \vert_{\infty} \leq  \dfrac{b^{n+1}}{R}.
\end{align}
\begin{theorem}\label{Main-T}
Let $u$ be a positive solution of IBVP \eqref{ibvp}. Let $\Omega'\subset\Omega$ be such that $u(x,0)=0$ for $x\in\Omega'$. Then for every ball $B_{R}(x_{0}) \Subset {\Omega}^{'}$ and every $R'\in(0,R)$
,\ there exist $ T' > 0 $ such that $u(x,t) = 0$ for $(x,t)\in B_{R'}(x_0)\times [0,T']$.
\end{theorem}
\begin{proof}
Let us choose $ b > 2 $ such that
\begin{align}
 \displaystyle \dfrac{b-2}{b-1} = \dfrac{R'}{R}   
\end{align}
We develop the iteration scheme $R_n, B_n, \theta_n$ as given in Definition \ref{R-teta}. Note that $R_\infty = R'$. We start by multiplying inequality in IBVP \eqref{ibvp} by $\theta_{n}^{2}$ and integrate over $\Omega \times (0,t)$,
\begin{align}
\int_{\Omega} \theta_{n}^{2}H(u)  \text{ } dx
+
\int_{0}^{t}\int_{\Omega} \nabla u \nabla(\theta_{n}^{2}F(u))  \text{ } dxd\tau \leq 0
\end{align}
Using \eqref{lemma-1-R} in Lemma \ref{lemma-1} in above we get
\begin{align}\label{grad-est}
\int_{\Omega}  \theta_{n}^{2}H(u)   \text{ } dx  +
\frac{1}{2} \int_{0}^{t}\int_{\Omega}| \nabla(\theta_{n}G(u))|^{2} \text{ } dxd\tau 
 \leq
(2C_{1}^2+1) \int_{0}^{t}\int_{B_{n}} G^{2}(u) |\nabla \theta_{n}|^2  \text{ } dxd\tau 
\vspace{- 0.8 cm}
\end{align}
In particular, $\nabla G(u) \in L^2_{loc}(\Omega\times [0,T])$. Now using \eqref{lemma-2-R} in Lemma \ref{lemma-2},
\begin{align}
\int_{\Omega}  &\theta_{n}^{2}H(u)  \text{ } dx  +
\frac{1}{2} \int_{0}^{t}\int_{\Omega}| \nabla(\theta_{n}G(u))|^{2} \text{ } dxd\tau\\
    \leq  
    & D \cdot(b^{2})^{n-1}
    t^{1-k(1+j)} 
    \left[\sup_{0\leq \tau \leq t}\int_{B_{n}}\theta_{n}^{2}H(u) dx + \int_{0}^{t} \int_{B_{n}} |\nabla(\theta_{n}G(u))|^{2} dx  d\tau\right]^{1+jk}  \nonumber
\end{align}
where $D \triangleq \dfrac{b^{4}}{R^{2}}(2C_{1}^{2}+1)C_2^{1-k}S^{k(1+j)}$. As $ 0 < t \leq T' $ we take supremum over $t$
\begin{multline}
\sup_{0 \leq \tau \leq T' } \int_{B_{n}} \theta_{n}^{2} H(u)  \text{ } dx
+
\int_{0}^{T'}\int_{B_{n}} |\nabla(\theta_{n}G(u)) |^{2}  \text{ } dxd\tau\\
 \leq  D ({b^2})^{n-1} (T')^{1-k(1+j)} \left[\sup_{0 \leq \tau \leq T' }\int_{B_{n}}\theta_{n}^{2}H(u) dx + \int_{0}^{T'} \int_{B_{n}} |\nabla(\theta_{n}G(u))|^{2} dx  d\tau\right]^{1+jk} \label{sup-eq}
\end{multline}
Note that $\gamma \triangleq \dfrac{1-(1+j)k}{k j}$ then
 $\gamma +1-(1+j) k = \gamma(1+k j)$. \\\\
Multiply above inequality by $ (T')^{\gamma}$. So \eqref{sup-eq}  yields,
\begin{align}
&(T')^{\gamma} \sup_{0 \leq \tau \leq T' } \int_{\Omega} \theta_{n}^{2}(x)H(u) \ dx
+
(T')^{\gamma}\int_{0}^{T}\int_{\Omega} |\nabla(\theta_{n}(x)G(u)) |^{2}  \text{ } dxd\tau \nonumber \\
&\leq 
D \cdot ({b^2})^{n-1} \left[(T')^{\gamma} \sup_{0 \leq \tau \leq T' }\int_{B_{n}}\theta_{n}^{2}(x)H(u) dx +(T')^{\gamma} \int_{0}^{T'} \int_{B_{n}} |\nabla(\theta_{n}(x)G(u))|^{2} dx  d\tau\right]^{1+jk} \nonumber\\
& \leq 
D \cdot ({b^2})^{n-1} \left[(T')^{\gamma} \sup_{0 \leq \tau \leq T' }\int_{B_{n-1}}\theta_{n-1}^{2}(x)H(u) dx +(T')^{\gamma} \int_{0}^{T'} \int_{B_{n-1}} |\nabla(\theta_{n-1}(x)G(u))|^{2} dx  d\tau\right]^{1+jk}  \nonumber
\label{T-gamma}
\end{align}
Let us introduce
\begin{align}
Y_{n} (T') \triangleq (T')^{\gamma} \sup_{0 \leq \tau \leq T' }\int_{B_{n}}\theta_{n}^{2}(x)H(u) dx +(T')^{\gamma} \int_{0}^{T'} \int_{B_{n}} |\nabla(\theta_{n}(x)G(u))|^{2} dx d\tau
\end{align}
then the above becomes
\begin{align}
    Y_{n}(T') \leq  D \cdot ({b^2})^{n-1} Y_{n-1}^{1+k j}(T')
\end{align}
By Ladyzhenskaya-Uraltceva iterative Lemma \ref{Lady-lemma},    if 
\begin{equation}\label{X_0-assump}
   Y_{0}(T')\leq D^{-\dfrac{1}{k j}} b^{-\dfrac{2}{k^2 j^2}} 
\end{equation}
then 
\begin{align}
Y_{n}(T')\rightarrow 0 \text{ whenever }  n \rightarrow \infty.
\end{align}
Since \eqref{grad-est} and $\gamma>0$, we can choose $T'$ small enough to satisfy \eqref{X_0-assump}. So the assertion holds.
\end{proof}
\begin{remark}
In fact, it was not necessary to assume for $u$ certain bounds and positively on $B_R(x_0)\times [0,T]$. It is enough to assume that non-negative solution belong to the class of bounded functions in $R^N\times [0,\infty],$ and then  
 using maximum principle, prove that it $u$ bounded  by initial data at  $R^N\times \{0\}$.
\end{remark}
\end{section}
\begin{section}{Example of functions $P$ and $F$ for which constrains holds} \label{Exmp-P-F}
Without loss of generality in this section we will assume that $0\leq u\leq 1$, and we will  show some generic examples of the function $P$ for which all constrains on the functions $F$, $G$, $H$ holds. Note that general case $0\leq u < M$ can be reduced with appropriate scaling in equation. Evidently this will lead to smaller time interval $[0,T]$ over which solution in $B(x_0,R)$ is be equal $0$.
\begin{remark}\label{P-test}
Let $P,I\in C^1(0,\infty)$ be as in Definition \ref{PIH}. Assume that
\begin{align}
    \limsup\limits_{s\to0} P(s)I(s) & < \infty,\\
    \limsup\limits_{s\to0} sP^{'}(s)I(s) & < \infty.
\end{align}
Then Assumption \ref{Assumps} hold, by Propositions \eqref{P-2}, 
\eqref{P-3} and Remark \ref{I-B-prop}. So Theorem \ref{Main-T} about finite speed of propagation holds.
\end{remark}
Below we will construct examples  of the function $P$ as in Remark \ref{P-test}.

\begin{exmpl}{$P(s)=s^{\gamma} \ ; \ s\in [0,\infty), \gamma > 0.$}. \normalfont
\begin{align}
 I(s) & =  \int_{s}^\infty t ^{-\gamma-1} d t =  \ \tfrac1\gamma s^{-\gamma}\\
 P(s)I(s) & = \tfrac1\gamma \\
 sP'(s)I(s) & = 1. 
\end{align}
\end{exmpl}
\begin{exmpl}{$ P(s)= \exp \left(-\dfrac{1}{s^{\gamma}}\right), s \in [0,1), \gamma > 0.\ P(0) = 0 \ ; P \in C([0,\infty))$}.\normalfont
\begin{align}
 I(s) = \int_{s}^{1} t^{-1}\exp \left(\dfrac{1}{t^{\gamma}}\right) dt .
\end{align} 
\begin{align}  \displaystyle
P(s)I(s) 
= &  {\exp\left(-\dfrac{1}{s^{\gamma}}\right)}\left[\int_{s}^{1} t ^{-1}\exp  \left(\dfrac{1}{t^{\gamma}}\right) dt \right]
\end{align}
By L'Hôpital's rule 
\begin{align} \displaystyle
\lim_{s \to 0} P(s)I(s) \equiv  \ \lim_{s \to 0} s^{\gamma} \equiv  \   \ 0,
\end{align}  verifies that exist $ \ \Lambda $ s.t. $ P(s)I(s) < \displaystyle \frac{\Lambda+1}{\Lambda} \  \forall \ s \in [0,1)$. 
Next using \eqref{R-4}
\begin{align}
    sI(s)P^{/}(s) =  \ &  \  s \left[\int_{s}^{1} t ^{-1}\exp \left(\dfrac{1}{t^{\gamma}}\right) dt \right] \exp \left(-\frac{1}{s^{\gamma}}\right)s^{\gamma-1}\\
 \lim_{s \to 0} sI(s)P^{/}(s)   =  & \lim_{s \to 0}
    \frac{ \displaystyle \left[\int_{s}^{1}t^{-1} \exp \left( \frac{1}{t^{\gamma}}\right) dt \right]s^{-\gamma}}{ \gamma \displaystyle \exp \left( \frac{1}{s^{\gamma}}\right)}\\
    \equiv & \ 1 +  \lim_{s \to 0}  \frac{ \left[\displaystyle \int_{s}^{1}t^{-1} \exp \left( \frac{1}{t^{\gamma}}\right) dt\right]}{ \displaystyle \exp \left( \frac{1}{s^{\gamma}}\right)}  \quad  \text{(by L'Hôpital's rule)} \\
    \equiv & \ 1 +  \lim_{s \to 0}  s^{\gamma} \quad \quad\quad \hspace{3 cm}\text{(by L'Hôpital's rule)}\\
     \equiv & \ 1
\end{align}
By \eqref{R-4} $
\lim_{s \to 0}{F(s)}[G(s)G^{/}(s)]^{-1}
\equiv  \  1 $ which verifies  existence of $\  C_{1}  >  0$ in  Assumption \eqref{A-1}.
\end{exmpl}
\begin{exmpl}{$ \displaystyle P(s)= \exp \left(-\int_{s}^{1} \frac{\zeta(\tau)}{\tau} d\tau\right), s \in (0,1] \ ;\ P(0) = 0 \ ; P \in C([0,\infty))$}. \normalfont Provided that $ \ 0 < k_{1} < \zeta(s) < k_{2}$.  \begin{align}
    \displaystyle I(s) =\int_{s}^{1} t^{-1}\exp \left(\int_{t}^{1} \frac{\zeta(\tau)}{\tau} d\tau\right) \ dt. 
    \end{align}
Then
\begin{align}
 \displaystyle P(s)I(s) 
     = & \  \frac{ \displaystyle\left[\int_{s}^{1} t^{-1}\exp \left(\int_{t}^{1} \frac{\zeta(\tau)}{\tau} d\tau\right) \ dt  \right]}
     {\displaystyle\left[ \exp \left(\int_{s}^{1} \frac{\zeta(\tau)}{\tau} d\tau\right)\right]} 
\end{align}
By L'Hôpital's rule
\begin{align}
   \lim_{s \to 0} P(s)I(s) \equiv & \  \lim_{s \to 0} \frac{1}{\zeta(s)}   \equiv  1  \label{ex-3}
      \end{align}
which verifies that $  \exists  \ \Lambda $ s.t. $ P(s)I(s) < \displaystyle \frac{\Lambda+1}{\Lambda} \  \forall \ s \in [0,1)$. Note that $\zeta(s) \equiv 1$. 
Next using \eqref{R-4}
\begin{align}
s I(s)P^{/}(s) = & \ {\displaystyle\left[\int_{s}^{1} t^{-1}\exp \left(\int_{t}^{1} \frac{\zeta(\tau)}{\tau} d\tau\right) \ dt \right] }{\displaystyle \exp \left(-\int_{t}^{1} \frac{\zeta(\tau)}{\tau} d\tau\right)\zeta(s)}\\
\lim_{s \to 0} s I(s)P^{/}(s) \equiv  \ & \lim_{s \to 0}  \frac{\displaystyle\left[\int_{s}^{1} t^{-1}\exp \left(\int_{s}^{1} \frac{\zeta(\tau)}{\tau} d\tau\right) \ dt \right]}{\displaystyle\exp \left(\int_{t}^{1} \frac{\zeta(\tau)}{\tau} d\tau\right)}
\end{align}
We apply L'Hôpital's rule to get
\begin{align}
\lim_{s \to 0} s I(s)P^{/}(s) \equiv \ \frac{1}{\zeta(s)}  \equiv \ 1  \end{align}
By \eqref{R-4} $
\lim_{s \to 0}{F(s)}[G(s)G^{/}(s)]^{-1}
\equiv  \  1 $ which verifies  $ \exists \  C_{1}  >  0$ in  Assumption \eqref{A-1}. 
\end{exmpl}
\begin{exmpl}{$ \displaystyle P(s)= \exp \left(-\int_{s}^{1} \frac{\zeta(\tau)}{\tau} d\tau\right), s \in (0,1] \ ;\ P(0) = 0 \ ; P \in C([0,\infty))$}\normalfont  with $ \displaystyle \lim_{\tau \to 0} \zeta(\tau) = \infty$. We will assume $\zeta$  be such that $\displaystyle 0 < \ k_{3} \leq \frac{\zeta}{\zeta_{0}} \leq k_{4} \ ; \text{for function} s.t. \zeta_{0}^{/} \leq 0 \ ; \ \sup_{0<s<1}\frac{s|\zeta_{0}^{/}|}{\zeta_{0}} =C_{0} < \infty$. Once again recollect 
\begin{align}
\displaystyle I(s) =\int_{s}^{1} t^{-1}\exp \left(\int_{t}^{1} \frac{\zeta(\tau)}{\tau} d\tau\right) \ dt. \end{align}
Similarly in\eqref{ex-3} we get
\begin{align}
\displaystyle  \lim_{s \to 0} P(s)I(s) \equiv \frac{1}{\zeta(s)}  = 0.
\end{align}
verifies the  existence of $ \ \Lambda $ s.t. $ P(s)I(s) < \displaystyle \frac{\Lambda+1}{\Lambda} \  \forall \ s \in [0,1)$. Note that $\zeta(s) \equiv \zeta_{0}(s) $.  \\\\
Next using \eqref{R-4} we will get
\begin{align}
sI(s)P^{/}(s) =  & s \left[\int_{s}^{1} t^{-1}\exp \left(\int_{t}^{1} \frac{\zeta(\tau)}{\tau} d\tau\right) \ dt\right] \exp \left(-\int_{t}^{1} \frac{\zeta(\tau)}{\tau} d\tau\right) \zeta(s) 
\end{align}
Moreover,
\begin{align}
\lim_{s \to 0} sI(s)P^{/}(s) = &  \lim_{s \to 0} \frac{\displaystyle \left[\int_{s}^{1} t^{-1}\exp \left(\int_{t}^{1} \frac{\zeta(\tau)}{\tau} d\tau\right) \ dt\right]    \zeta_{0}(s)} {\displaystyle \exp \left(\int_{s}^{1} \frac{\zeta(\tau)}{\tau} d\tau\right)}\nonumber\\
\equiv   & \ 1 + \lim_{s \to 0} \frac{\left[\displaystyle \int_{s}^{1} t^{-1}\exp \left(\int_{t}^{1} \frac{\zeta(\tau)}{\tau} d\tau\right) \ dt \right] \frac{ \displaystyle s \vert \zeta_{0}^{/}(s)\vert}{ \displaystyle \zeta(s)}}{\displaystyle \exp \left(\int_{s}^{1} \frac{\zeta(\tau)}{\tau} d\tau\right) \ dt} \quad  \text{(by L'Hôpital's rule)}\nonumber \\
\equiv   & \ 1 +  \lim_{s \to 0}  \frac{C_{0}}{\zeta(s)} \hspace{ 6 cm } \text{(by L'Hôpital's rule)}\nonumber\\
\equiv   & \ 1
\end{align}
By \eqref{R-4} $
\lim_{s \to 0}{F(s)}[G(s)G^{/}(s)]^{-1}
\equiv  \  1 $ which verifies existence of $\  C_{1}  >  0$ in  Assumption \eqref{A-1}. 
\end{exmpl}
\end{section}
\begin{section}{Auxiliary properties of Functional spaces, and a priory estimates for the solution }
\begin{theorem} \label{dual-inclusion}
Let $X \subset Y$ be Banach spaces. Let $X^{*}$ and $Y^{*}$ be the dual spaces of $X$ and $Y$ respectively. Then $Y^{*} \subset X^{*}$.
\end{theorem}
\begin{proof}
Let $f \in Y^{*}$. Then $f(x)$ is well defined for all $x\in X \subset Y$. Moreover 
\begin{align}
\sup_{\vert\vert x \vert \vert_{X}=1} f(x) 
&\leq
 \sup_{\vert\vert x \vert \vert_{X}=1} \vert\vert f \vert \vert_{Y^{*}} \vert\vert x \vert \vert_{Y}
\leq \vert\vert f \vert \vert_{Y^{*}} \implies
\vert\vert f \vert \vert_{X^{*}} \leq \vert\vert f \vert \vert_{Y^{*}}. \qedhere
\end{align}
\end{proof}
\begin{theorem}
Let $X,Y$ be Banach spaces and $Z$ be a topological vector space  such that $X,Y \subset Z$. Let $\displaystyle  X+Y \triangleq  \{z=x+y : x\in X, y \in Y \}$ with
$ \displaystyle
\vert \vert z\vert \vert_{X+Y} =  \inf \{\vert \vert x\vert \vert_{X} + \vert \vert y\vert \vert_{Y} :   x\in X, y \in Y, x+y =z \}.$ Then
\begin{enumerate}
\item [{\normalfont i.}]  $X+Y$ is a Banach space.
\item [{\normalfont ii.}]  $X,Y \subset X+Y$.
\item [{\normalfont iii.}]  $(X+Y)^{*} = X^{*} \cap Y^{*}$
\end{enumerate}
\end{theorem}
\begin{enumerate}
    \item  
\begin{proof}
    Let $z_{n} \in X+Y$ be such that 
$ \displaystyle
\sum_{n} \vert\vert z_{n}\vert\vert_{X+Y} < \infty
$. For every $n$, let $x_{n}\in X$, $y_{n}\in Y$ be such that $  x_{n}+y_{n} =  z_{n}$ and
\begin{align}
\vert\vert x_{n}\vert\vert_{X} + \vert\vert y_{n}\vert\vert_{Y} <  \vert\vert z_{n}\vert\vert_{X+Y} + 2^{-n}.
\end{align}
Then $ \sum_{n} \vert\vert x_{n} \vert\vert_{X} + \sum_{n} \vert\vert y_{n} \vert\vert_{Y} < \sum_{n} \vert\vert z_{n} \vert\vert_{X+Y}  + 2^{-n}$.
Since $X$ and $Y$ are Banach spaces there are $x\in X$ and $y \in Y$ such that $  \displaystyle    x = \sum_{n} x_{n} \ \   y = \sum_{n} y_{n}$. Thus, $\displaystyle \sum_{n} z_{n} =  x+y.$ \qedhere
\end{proof}
\item
\begin{proof}
    For $x \in X$ then $x \in X+Y$, and by definition
    $ \vert\vert x\vert\vert_{X+Y}  \leq \vert\vert x \vert \vert_{X}$. Similarly,
    \[ \vert\vert y\vert\vert_{X+Y}  \leq \vert\vert y \vert \vert_{Y}. \qedhere\]
    \end{proof}
\item
\begin{proof}
Since $X \subset X+Y$ and $y \subset X+Y$, it follows from Theorem \eqref{dual-inclusion}
\[ (X+Y)^{*} \subset X^{*} \  \text{ and }  \ (X+Y)^{*} \subset Y^{*}                        \]
So that $ (X+Y)^{*} \subset ( X^{*} \cap Y^{*})$.
\noindent Now show the converse. Let
$f \in X^{*} \cap Y^{*}$, $z\in X+Y$. For any $n \in \mathbb{N}$, let $x_{n}\in X$ and $y_{n}\in Y$ be such that $ z =  x_{n} + y_{n} $ and
\begin{align}
    \sum_{n} \vert\vert x_{n} \vert\vert_{X} + \sum_{n} \vert\vert y_{n} \vert\vert_{Y} < \sum_{n} \vert\vert z_{n} \vert\vert_{X+Y}  + \frac{1}{n}.
\end{align}
Then $ f(z) =  f(x) + f(y) $ is well-defined and 
\begin{align}
\vert f(z) \vert
\leq &
    \vert\vert f \vert\vert_{X^{*}}  \vert\vert x_{n} \vert\vert_{X} +  \vert\vert f \vert\vert_{Y^{*}} \vert\vert y_{n} \vert\vert_{Y}\\
\leq &
    \max\{\vert\vert f \vert\vert_{X^{*}},\vert\vert f \vert\vert_{Y^{*}} \} \ ( \vert\vert x_{n} \vert\vert_{X} +   \vert\vert y_{n} \vert\vert_{Y})\\
\leq &
     \vert\vert f \vert\vert_{X^{*} \cap Y^{*}}\left( \vert\vert z \vert\vert_{X+Y} + \frac{1}{n}\right)
\end{align}
As $n \rightarrow \infty $,  $|f(z)| \leq  \vert\vert f \vert\vert_{X^{*} \cap Y^{*}} \ \vert\vert z \vert\vert_{X+Y} $.
So $f \in (X+Y)^{*}$  \text{ and } 
$  \vert\vert f \vert\vert_{(X+Y)^{*}} \leq \vert\vert f \vert\vert_{X^{*} \cap Y^{*}} .\qedhere$
\end{proof}
\end{enumerate}
\begin{theorem}\label{compact}
Let $B_{0} \Subset B_{1} \subset B_{2}$ be Banach Spaces. Let $I \subset \mathbb{R}$ and $\mathbb{F} \subset L^{P}(I,B_{0})$ such that $ \displaystyle
\sup_{f\in \mathbb{F}} \vert\vert f \vert\vert_{L^{P}(I,B_{0})} < \infty \ , \
\sup_{f\in \mathbb{F}} \vert\vert f^{'} \vert\vert_{L^{1}(I,B_{2})} < \infty
$. Then $ \mathbb{F} $ is compact in $L^{P}(I,B_{1})$.
\end{theorem}
 See corollary 4 of Theorem 5 in \cite{cpmct-SIM}
 \section{Weak viscous solution of degenerate Einstein equation. }\label{viscous-solution}

\begin{theorem} \label{grad-u-bound}
Let $u_{\epsilon}(x,t)\in C^{2,1}_x(\bar{\Omega}\times (0,T]) \ ;\ 0 < \epsilon \leq u_{\epsilon} \leq K < \infty$ be classical solution of the $IBVP_{\epsilon}$ \eqref{ibvp-ep} with initial data on $W^{1,2}(\Omega)$. Let 
$ \displaystyle \tilde{H}(u_{\epsilon})=\int_{\epsilon}^{u_{\epsilon}} \sqrt{{h_{\epsilon}(s)}/{F_{\epsilon}(s)}} \ ds.$
Then for any $\tau>0$,
\begin{align}\label{energy-id}
\int_{0}^{\tau} \int_{\Omega}\left(\frac{\partial\tilde{H}(u_{\epsilon}(x,t))}{\partial t}\right)^2 dxdt 
+ \int_\Omega|\nabla u_{\epsilon}(x,\tau)|^2 dx
= \int_{\Omega} |\nabla g(x)|^2 dx
\end{align}
%
and
\begin{multline}
\int_0^{T} \int_{\Omega}\left(\frac{\partial\tilde{H}(u_{\epsilon}(x,t))}{\partial t}\right)^2 dxdt 
 +
\int_0^{T}  \int_\Omega |\nabla u_{\epsilon}(x,t)|^2 \ dxdt. \\ 
+ \int_{0}^{T} \int_{\Omega} \vert u_{\epsilon}(x,t)\vert^{2} \ dxdt 
\leq 
C(\Omega,T) \int_\Omega |\nabla g(x)|^2 \ dx.
 \label{grad-u-bound-result}
\end{multline} 
\end{theorem}
\begin{proof}
Recall that by \eqref{H-F-def} we write $\displaystyle H(u_{\epsilon}) \triangleq \int_{\epsilon}^{u_{\epsilon}} h(s) \ ds $.
First observe that if $u_{\epsilon}(x,t)$ is solution of \eqref{ibvp-ep} then one can have the following:
\begin{align}\label{h-eq}
\frac{h(u_{\epsilon})}{F(u_{\epsilon})}(u_{\epsilon})_{t}-\Delta u_{\epsilon}= & 0   \hspace{1.4 cm } \text{on} \quad \Omega \times (0,T), \\
u_{\epsilon}(x,0) = & \epsilon + g(x) \hspace{0.4 cm } \text{on} \quad \Omega,\\
u_{\epsilon}(x,t) = & \epsilon \cdot \psi(x) \quad \text{on} \quad \partial \Omega \times (0,T). \label{BC-3}  
\end{align}
with some $g(x) \geq 0, \ g(x) \in W^{1,2}(\Omega)$ and $\psi(x) \geq 0$. Multiply equation \eqref{h-eq} by ${(u_\epsilon)}_{t}$ and integrate over $\Omega\times (0,\zeta_{1})$ for $0 < \zeta_{1} \leq T$.
\begin{align}
\int_0^{\zeta_{1}} \int_{\Omega}\left(\frac{\partial\tilde{H}(u_{\epsilon}(x,t))}{\partial t}\right)^2dxdt -
\int_0^{\zeta_{1}} \int_{\Omega} (u_{\epsilon})_{t}\Delta u_{\epsilon}
 & = 0 \label{post-int}
\end{align}
By using \eqref{BC-3} we get 
\begin{align}
     \int_{\Omega} (u_{\epsilon})_{t}\Delta u_{\epsilon} = \int_{\partial \Omega} (u_{\epsilon})_{t}\nabla u_{\epsilon} \cdot \mu^{i} ds - \int_{\Omega} (\nabla u_{\epsilon})_{t} \nabla u_{\epsilon} \ dx  = -\frac{1}{2} \int_{\Omega} \left({\vert \nabla u_{\epsilon} \vert^{2}}\right)_{t} \ dx 
\end{align}
Integrate over time
\begin{align}
      \int_{0}^{\zeta_{1}}\int_{\Omega} (u_{\epsilon})_{t}\Delta u_{\epsilon} \ dxdt= -\frac{1}{2} \int_{\Omega} \vert 
      \nabla u_{\epsilon}(x,\zeta_{1}) \vert^{2} \ dx
      + \frac{1}{2} \int_{\Omega} \vert 
      \nabla g(x) \vert^{2} \ dx \label{est-1}
\end{align}
Using \eqref{est-1} in \eqref{post-int} we get
\begin{equation}
\int_0^{\zeta_{1}} \int_{\Omega}\left(\frac{\partial\tilde{H}(u_{\epsilon}(x,t))}{\partial t}\right)^2dxdt
+
\frac{1}{2}\int_\Omega|\nabla u_{\epsilon}(x,\zeta_{1})|^2dx =
\frac{1}{2} \int_\Omega |\nabla g(x)|^2dx \ ; \ 0 < \zeta_{1} \leq T \label{result-1}
\end{equation}
As $u \in W^{1,2}(\Omega) $ by integrating  Friedrich's inequality in time over $(0,\zeta_{1})$
\begin{align}
  \int_{0}^{\zeta_{1}} \int_{\Omega} (u_{\epsilon}(x,t))^{2} \ dxdt - C_{F}(\Omega) \epsilon^{2}\zeta_{1}\int_{\partial\Omega} \vert \psi(x) \vert^{2} \ ds 
  \leq
  C_{F}(\Omega) \int_{0}^{\zeta_{1}} \int_{\Omega} \vert \nabla u_{\epsilon}(x,t) \vert^{2} \ dxdt
   \label{F-D-in}
\end{align}
By \eqref{result-1}
\begin{align}
    \int_0^{\zeta_{1}} \int_{\Omega}\left(\frac{\partial\tilde{H}(u_{\epsilon}(x,t))}{\partial t}\right)^2dxdt 
    \leq  \int_{\Omega}  |\nabla g(x)|^2dx \label{est-2}
\end{align}
and 
\begin{align}
  \int_{0}^{\zeta_{1}}  \int_\Omega|\nabla u_{\epsilon}(x,t)|^2 dx dt 
  \leq
  \zeta_{1}  \int_{\Omega}  |\nabla g(x)|^2dx  \ ; \ 0 < \zeta_{1} \leq T \label{est-3}
\end{align}
By adding  \eqref{est-2} and \eqref{est-3} we get
\begin{multline}
  \int_0^{\zeta_{1}} \int_{\Omega}\left(\frac{\partial\tilde{H}(u_{\epsilon}(x,t))}{\partial t}\right)^2dxdt 
 +
 \int_0^{\zeta_{1}} \int_\Omega |\nabla u_{\epsilon}(x,t)|^2\ dxdt
\leq 
(1+{\zeta_{1}})  \int_\Omega |\nabla g(x)|^2 \ dx 
\end{multline}
Let $0 < \epsilon_{0} < 1 $ be fixed. Using \eqref{F-D-in} in above
\begin{multline}
\int_0^{\zeta_{1}} \int_{\Omega}\left(\frac{\partial\tilde{H}(u_{\epsilon}(x,t))}{\partial t}\right)^2dxdt 
 +
\epsilon_{0} \int_0^{\zeta_{1}} \int_\Omega |\nabla u_{\epsilon}(x,t)|^2\ dxdt + \frac{(1-\epsilon_{0})}{{C_{F}}(\Omega)}
 \int_{0}^{\zeta_{1}} \int_{\Omega} \vert u_{\epsilon}(x,t) \vert^{2} \ dxdt \nonumber \\
\leq 
(1+{\zeta_{1}})  \int_\Omega |\nabla g(x)|^2 \ dx  + (1-\epsilon_{0}){\zeta_{1}} \epsilon^{2} \int_{\partial\Omega} \vert \psi(x) \vert^{2} \ ds  \ ; \   \zeta_{1} \in (0,T] 
\end{multline}
Setting $\zeta_{1}=T$,
\begin{multline}
  \int_0^{T} \int_{\Omega}\left(\frac{\partial\tilde{H}(u_{\epsilon}(x,t))}{\partial t}\right)^2dxdt 
 +
 \int_0^{T}  \int_\Omega |\nabla u_{\epsilon}(x,t)|^2\ dxdt + 
 \int_{0}^{T} \int_{\Omega} \vert u_{\epsilon}(x,t) \vert^{2} \ dxdt \\
\leq 
C(\Omega,T) \left(\int_\Omega |\nabla g(x)|^2 \ dx 
+
 \epsilon^{2} \int_{\partial\Omega} \vert \psi(x) \vert^{2} \ ds \right)
\end{multline}
where $ \displaystyle C(\Omega,T) \triangleq  
{[\max(1+T ,(1-\epsilon_{0})T)}]/{[\min(\epsilon_{0},{(1-\epsilon_{0})}/{C_{F}(\Omega)})]} \ ; \ 0 < \epsilon_{0} < 1$. As $\epsilon \rightarrow 0$ it follows the result
\begin{multline}
\int_0^{T} \int_{\Omega}\left(\frac{\partial\tilde{H}(u_{\epsilon}(x,t))}{\partial t}\right)^2 dxdt 
 +
\int_0^{T}  \int_\Omega |\nabla u_{\epsilon}(x,t)|^2 \ dxdt  \\ 
+ \int_{0}^{T} \int_{\Omega} \vert u_{\epsilon}(x,t) \vert^{2} \ dxdt 
\leq 
C(\Omega,T) \int_\Omega |\nabla g(x)|^2 \ dx  \qedhere
\end{multline} 
\end{proof}
\begin{remark} \label{grad-u-L2}
Note that $\displaystyle {\partial\tilde{H}(u)}/{\partial t}$ is uniformly bounded in $L^{2}(\Omega \times (0,T))$ and $u$ is uniformly bounded in $L({I, W^{1,2}(\Omega)})$
\end{remark}
\begin{theorem} \label{grad-G-bound}
Let $u_{\epsilon}$ be a family of strong solutions to $IBVP_{\epsilon}$ \eqref{ibvp-ep} such that ${\partial H(u_{\epsilon})}/{\partial t}, \Delta u_{\epsilon} \in L^{2}_{loc}(\Omega \times (0,T))$, satisfying the estimate $ \displaystyle
0 < \epsilon \leq u_{\epsilon} \leq K < \infty$. Then
for every $\theta \in C_{c}^{1}(\Omega)$, $u_{\epsilon}$ satisfy the following uniform estimate
\begin{multline}
     \int_{\Omega \times (0,T)} |\nabla (\theta G(u_{\epsilon}))|^{2} \ dxdt
    \leq
    c \int_{\Omega} \theta^{2} |G(u_{\epsilon}(x,0))|^2 \ dx +
    c \int_{\Omega \times (0,T)} |\nabla \theta|^{2} (G(u_{\epsilon}))^{2} \ dxdt. 
    \label{1-star}
\end{multline} 
Thus $G(u_{\epsilon})$ is uniformly bounded in $L^{2}(I, W^{1,2}_{loc}(\Omega)) $ for each $\epsilon$.
\end{theorem}
\begin{proof}
Let $\Omega_{T} \triangleq \Omega \times (0,T)$. Multiply both sides of the inequality \eqref{ibvp-ep}  by $\theta^{2}$ and integrate over $\Omega_{T}$ we get,   
\begin{align}
\int_{\Omega}[H(u_{\epsilon}(x,T))-H(u_{\epsilon}(x,0))]\theta^{2} \ dx 
& = - \int_{\Omega_{T}} ( \nabla u_{\epsilon}) ^{2} F^{'}(u_{\epsilon}) \theta^{2} \ dxdt
- \int_{\Omega_{T}} F(u_{\epsilon}) \nabla u_{\epsilon} \cdot \nabla (\theta^{2}) \ dxdt \nonumber
\end{align}
Rearrange above and using (i) in Proposition \eqref{P-1}, 
\begin{align}
    \int_{\Omega}H(u_{\epsilon 0})\theta^{2} \ dx 
    +
    \int_{\Omega_{T}} \vert F(u_{\epsilon}) \nabla u_{\epsilon} \cdot \nabla (\theta^{2})\vert \ dxdt
& \geq 
  \int_{\Omega_{T}} \vert \nabla u_{\epsilon} G^{'}(u_{\epsilon})\vert^{2} \theta^{2}
 \ dxdt \label{after-asump-2}
\end{align}
Let $0 < \epsilon < \frac{1}{2} $ be fixed. By Cauchy Inequality 
\begin{align}
\vert \nabla G(u_{\epsilon}) \vert^{2} \theta^{2} 
\geq &
(1- 2 \epsilon) \vert \nabla (\theta G(u_{\epsilon})) \vert^{2}
-
\left(\frac{1}{2\epsilon} - 1\right)(G(u_{\epsilon})\nabla \theta)^{2}  \label{rhs-1-ineq}
\end{align}
using \eqref{rhs-1-ineq} in \eqref{after-asump-2} we get
\begin{multline}
\int_{\Omega}H(u_{\epsilon 0})\theta^{2} \ dx 
+
\left(\frac{1}{2\epsilon} - 1\right) \int_{\Omega_{T}}
\vert G(u_{\epsilon})\nabla \theta)\vert ^{2}  \ dxdt \\
 +
 2 \vert \vert \nabla \theta \vert \vert_{L^{\infty}} \int_{\Omega_{T} } \vert F(u_{\epsilon}) \vert
 \vert 
 \nabla u_{\epsilon}\vert 
 \ dxdt
\geq
(1- 2 \epsilon)\int_{\Omega_{T}}
 \vert \nabla (\theta G(u_{\epsilon})) \vert^{2} \ dxdt
 \label{after-cauchy-2} \qedhere
\end{multline}
\end{proof}
\begin{remark}
Note that $ F(u_\epsilon) \in L^{2}_{\Omega_{T}}$. By Remark \ref{grad-u-L2}, $ \nabla (u_\epsilon) \in L^{2}_{\Omega_{T}}$. Then \eqref{after-cauchy-2} yields
\begin{multline}
\vert \vert H(u_{\epsilon 0})\vert \vert_{L^{1}_{\Omega}} 
+
\vert \vert G(u_{\epsilon})\nabla \theta \vert \vert_{L^{2}_{\Omega_{T}}}^{2}  
 +  \vert \vert F(u_{\epsilon})\vert \vert_{L_{\Omega_{T}}^{2}}
 \cdot
 \vert \vert
 \nabla u_{\epsilon}\vert \vert_{L_{\Omega_{T}}^{2}}
\geq
K_{3}\vert \vert \nabla (\theta G(u_{\epsilon})) \vert \vert^{2}_{L^{2}_{\Omega_{T}}}  
\end{multline}
where $ \displaystyle K_{3} \triangleq {(1- 2 \epsilon)}/ [{\max \{ ({1}/{2\epsilon}) - 1\ , \ {2b^{n+1}}/{R} \}}]$ ; $0 < \epsilon < \frac{1}{2} $. Here $\mathbf{1}_{\supp \theta}H(u_{\epsilon}) \in L^{1}(\Omega)$. Thus $G(u_{\epsilon})$ are uniformly bounded in $L^{2}(I, W^{1,2}_{loc}(\Omega)) $.
\end{remark}

\begin{proposition}\label{H-bound-1}
\normalfont
Assume all conditions in Proposition \eqref{P-1} and  Proposition \eqref{P-2}.\\
If \[ \displaystyle
\liminf_{s\rightarrow 0} {P(s)\left(\int_{s}^{M} \frac{1}{\sigma P(\sigma)} \ d\sigma \right)^{\frac1\Lambda}} > 0 \]
then 
\[\sup\limits_{\epsilon>0}\int_{I}\int_{\Omega}\theta|\nabla H(u_{\epsilon})|^2\leq C_\theta<\infty \] for every $\theta\in C^1_c(\Omega)$.
\end{proposition}
\begin{proof}
Note that
$ \displaystyle
    |\nabla H(u_{\epsilon})| 
    \leq
    \left\vert \frac{H'(u_{\epsilon})}{G'(u_{\epsilon})} \right \vert |\nabla G(u_{\epsilon})|
$.  Hence, by Theorem \ref{grad-G-bound}, it suffices to show that $\displaystyle \sup_{0<s<K} \left\vert \frac{H'(s)}{G'(s)} \right\vert <\infty.
$ For this end it suffices to verify
$ \displaystyle    \limsup_{s \rightarrow 0} \left\vert \frac{H'(s)}{G'(s)} \right\vert < \infty.
$\\\\
Recall (i) in Proposition \ref{P-1}. By \eqref{F-I}  we have
\begin{align}
F(s)     = &  \Lambda^{-\frac{1}{\Lambda}-1}s^{-1}(I(s))^{-\frac{1}{\Lambda}-1}\\
F^{'}(s) = & B s^{-2} {(I(s))^{-\frac{1}{\Lambda}-2}}({P(s)})^{-1}\left(\frac{1+\Lambda}{\Lambda} -P(s)I(s)\right) \ ; \ B\triangleq   \Lambda^{-\frac{1}{\Lambda}-1} \\
{G^{'}(s)} = & \sqrt{B} s^{-1} {(I(s))^{-\frac{1}{2\Lambda}-1}}({P(s)})^{-\frac{1}{2}}\sqrt{\frac{1+\Lambda}{\Lambda} - P(s)I(s)}  \label{F-deri-I}
\end{align}
Next we write $H(s)$ in term of $I(s)$. By \eqref{I-s} in \eqref{H-Result}
\begin{align}
    H(s)    = & (\Lambda I(s))^{-\frac{1}{\Lambda}}\\
    H^{'}(s)= & B {(I(s))^{-\frac{1}{\Lambda}-1}}s^{-1} (P(s))^{-1} \label{H-deri-I}
\end{align}
Using \eqref{H-deri-I} and \eqref{F-deri-I}
\begin{align}
    \left\vert \frac{H'(s)}{G'(s)} \right\vert 
    =\frac{1}{\sqrt{B}} \frac{1}{\sqrt{(I(s))^{\frac{1}{\Lambda}}P(s)}}\frac{1}{\sqrt{\frac{1+\Lambda}{\Lambda}-I(s)P(s)}} \label{ratio-1}
\end{align}
By Proposition \ref{P-2} we have 
$ \displaystyle \frac{1+\Lambda}{\Lambda} > I(s)P(s) > 0 \ ; \ s\in [0,M). $\\\\
Thus, by \eqref{ratio-1}  $ \displaystyle
    \lim_{s\rightarrow 0} \left\vert \frac{H'(s)}{G'(s)} \right\vert 
    < \infty $
whenever $
\lim_{s\rightarrow 0} \inf{P(s) (I(s))^{\frac1\Lambda}} > 0$. 
\end{proof}
\begin{theorem}
Let $u_{\epsilon}$ be a family of strong solutions to \eqref{ibvp-ep} that satisfies all conditions in Theorem \eqref{grad-G-bound}. Let $\Omega^{'}_{I} \triangleq \Omega^{'} \times I$ \ ; \ $\Omega'\Subset\Omega$.  Then  $u_{\epsilon}$ holds the following uniform estimates 
\begin{align}
\left \vert\int_{I}\int_{\Omega}  \Psi \frac{\partial  H(u_{\epsilon})}{\partial t} \ dxdt   \right \vert
\leq C(\Omega',K)\left(\vert\vert \Psi   \vert\vert_{L^{\infty}(\Omega\times I)}
+
\vert\vert   \nabla \Psi   \vert\vert_{L^2(\Omega\times I)}
\right), \label{L1-bound}
\end{align}
and
\begin{align}
\int_{I}\int_{\Omega'}|\nabla H(u_{\epsilon})|^2\ dxdt
\leq C(\Omega',K). 
\end{align}
for each $\Psi\in C^1_c(\Omega^{'}_{I})$  and for each $\Omega'\Subset\Omega$.
\end{theorem}
\begin{proof}
Multiply both sides of the inequality in  \eqref{ibvp-ep}  by $ \Psi$ and integrate over $\Omega^{'}_{I}$ we get,  
\begin{align}
\left\vert \int_{\Omega^{'}_{I}} \Psi \frac{\partial H(u_\epsilon)}{\partial t} \ dxdt  \right\vert 
 &\leq  
 \int_{\Omega^{'}_{I}} \left\vert \nabla u_{\epsilon} \cdot \nabla (F( u_{\epsilon})\Psi)\right\vert \ dxdt \nonumber \\
& \leq  
  \int_{\Omega^{'}_{I}} \left\vert F(u_{\epsilon})
 \nabla u_{\epsilon}\cdot\nabla \Psi \right\vert
 + 
 \left\vert \Psi F^{'}( u_{\epsilon}) (\nabla u_{\epsilon})^{2}
 \right\vert \ dxdt \label{z-1}  \nonumber \\
& \leq 
C_{2}\int_{\Omega^{'}_{I}}  \left\vert G(u_{\epsilon}) G^{'}(u_{\epsilon})
   \nabla u_{\epsilon}\cdot\nabla \Psi \right\vert
   \ dxdt
   + 
\int_{\Omega^{'}_{I}}   \left\vert \Psi (G^{'}( u_{\epsilon}))^2 (\nabla u_{\epsilon})^{2} \right\vert \ dxdt \nonumber
\end{align}
Compute the right hand side of above
\begin{align}
    & \leq 
   C_{2}\vert \vert G(u_{\epsilon}) \vert \vert_{L^{\infty}(\Omega^{'}_{I})} \int_{\Omega^{'}_{I}}  \vert \nabla G(u_{\epsilon}) \vert\vert \nabla \Psi \vert  \ dxdt
   + 
   \vert\vert \Psi \vert \vert_{L^{\infty}(\Omega^{'}_{I})} \int_{\Omega^{'}_{I}}  \vert \nabla G(u_{\epsilon})\vert^2
   \ dxdt \nonumber\\
& \leq  
C_{2}\vert \vert G(u_{\epsilon}) \vert \vert_{L^{\infty}(\Omega^{'}_{I})}  \cdot \vert\vert \mathbf{1}_{\supp \Psi} \nabla G(u_{\epsilon}) \vert \vert_{L^{2}(\Omega^{'}_{I})} \cdot
\vert \vert \nabla \Psi \vert \vert_{L^{2}(\Omega^{'}_{I})} 
+
\vert\vert \Psi \vert \vert_{L^{\infty}(\Omega^{'}_{I})} \cdot
\vert  \vert \mathbf{1}_{\supp \Psi} \nabla G(u_{\epsilon}) \vert \vert_{L^{2}(\Omega^{'}_{I})}^{2}
\nonumber\\
& \leq  
C_{2}\vert \vert G(u_{\epsilon}) \vert \vert_{L^{\infty}(\Omega^{'}_{I})}  \cdot \vert\vert \mathbf{1}_{\supp \Psi} \nabla G(u_{\epsilon}) \vert \vert_{L^{2}(\Omega^{'}_{I})} \cdot
(\vert \vert \nabla \Psi \vert \vert_{L^{2}(\Omega^{'}_{I})} +\vert\vert \Psi \vert \vert_{L^{\infty}(\Omega^{'}_{I})} ) \nonumber \\
& \quad  \quad  \quad \quad  \quad  \quad \quad  \quad  \quad   + 
(\vert \vert \nabla \Psi \vert \vert_{L^{2}(\Omega^{'}_{I})} +\vert\vert \Psi \vert \vert_{L^{\infty}(\Omega^{'}_{I})} ) \cdot
\vert  \vert \mathbf{1}_{\supp \Psi} \nabla G(u_{\epsilon}) \vert \vert_{L^{2}(\Omega^{'}_{I})}^{2} \nonumber\\
& \leq  
C_{2}\left(\vert \vert \nabla \Psi \vert \vert_{L^{2}(\Omega^{'}_{I})} +\vert\vert \Psi \vert \vert_{L^{\infty}(\Omega^{'}_{I})} \right) \cdot
\left(\vert \vert G(u_{\epsilon}) \vert \vert_{L^{\infty}(\Omega^{'}_{I})}  \cdot \vert\vert  \nabla G(u_{\epsilon}) \vert \vert_{L^{2}(\Omega^{'}_{I})} +
\vert  \vert  \nabla G(u_{\epsilon}) \vert \vert_{L^{2}(\Omega^{'}_{I})}^{2}\right) \nonumber
\end{align}
Recall that $0 < \epsilon \leq u_{\epsilon} \leq K < \infty $ and $G(u_{\epsilon})$ are uniformly bounded. 
\begin{align}
\therefore C_{2}\vert \vert G(u_{\epsilon}) \vert \vert_{L^{\infty}(\Omega^{'}_{I})}  \cdot \vert\vert  \nabla G(u_{\epsilon}) \vert \vert_{L^{2}(\Omega^{'}_{I})} +
\vert  \vert  \nabla G(u_{\epsilon}) \vert \vert_{L^{2}(\Omega^{'}_{I})}^{2} \leq C(\Omega^{'},K)
\end{align}
which follows the result $\eqref{L1-bound}$ \qedhere
\end{proof}
From above theorem follows important statement
\begin{corollary}\label{proposition on compactness}
If we obtain an estimate analogous to \eqref{1-star}, replacing $\nabla G(u_{\epsilon})$ with $\nabla H(u_{\epsilon})$,  then we can apply Theorem \eqref{compact} to conclude compactness of $\{ H(u_{\epsilon})\}$ in $L^{2} (I,L^{q}_{loc}(\Omega))$ with $ q <\frac{2N}{N-2} $. Since $H:\mathbb{R}\to \mathbb{R}$ is a homeomorphism, and $u_\epsilon$ is uniformly bounded, it follows that  
$\{u_{\epsilon}\}$ is compact in $L^{q}_{loc} (\Omega\times I)$, $q\geq1$.
\end{corollary}
\begin{itemize}
\item 
Estimate \eqref{1-star} implies $G(u_{\epsilon})$ are uniformly bounded in $L^{2}(I, W^{1,2}_{loc}(\Omega)) $.
\item
Estimate \eqref{H-t-1} together with \eqref{1-star}  implies that ${\partial H(u_{\epsilon})}/{\partial t}$ is uniformly bounded in $L^{1}(I,L_{loc}^{1}(\Omega)) + L^{2}(I,W_{loc}^{-1,2}(\Omega))$ by  Theorem \ref{interpolation*}(3).
\item
To obtain for  sufficient conditions under which last Corollary \ref{proposition on compactness} holds w.r.t.  input functions, it is suffice to have  Proposition \ref{H-bound-1}.
\end{itemize}
\end{section}

\begin{section}{Appendix}
In the prove of the main theorem \ref{Main-T} on localisation we used basic iterative inequality and conclusion from it, which explicitly formulated by Lady\v{z}enskaja - Solonnikov- Ural'ceva in \cite{LAU}.
\begin{Lad-Ur}\label{Lady-lemma}
Let sequence $y_n \ ; \ n=0,1,2,...$ be non-negative, satisfying the recursion inequality
$ \displaystyle  y_{n+1}\leq c\text{ }b^n \text{ }y_n^{1+\epsilon} $ with some constants $ c ,\epsilon > 0 \text{ and } b\geq 1$. Then
\[ \displaystyle y_n \leq c^{\frac{(1+\epsilon)^n -1}{\epsilon}}\text{ } b^{\frac{(1+\epsilon)^n -1}{\epsilon^2} -\frac{n}{\epsilon}}\text{ }y_0^{(1+\epsilon)^n}.\]
In particular $ \displaystyle \text{if} \quad y_0 \leq \theta_L = c^\frac{-1}{\epsilon} \text{ } b^\frac{-1}{\epsilon^2} \text{ and } {b > 1}  \text{ then } y_n \leq \theta\text{ } b^{\frac{-n}{\epsilon}}$ 
and consequently \[ \displaystyle y_n \rightarrow 0 \text{ when } n\rightarrow \infty.\]
\end{Lad-Ur}

\bibliography{PAPER-1}
 \bibliographystyle{abbrv}
\end{section}
\end{document}